\documentclass[11pt,reqno]{amsart}

\usepackage[T1]{fontenc}
\usepackage[utf8]{inputenc}
\usepackage{amsmath,amsfonts,amssymb,amsthm}
\usepackage{mathrsfs}
\usepackage[margin=1.1in]{geometry}
\usepackage[colorlinks=true,linkcolor=blue,citecolor=blue,urlcolor=blue]{hyperref}

\theoremstyle{plain}
\newtheorem{theorem}{Theorem}
\newtheorem{lemma}{Lemma}
\newtheorem{corollary}{Corollary}

\newtheorem{thmA}{Theorem}

\theoremstyle{definition}
\newtheorem{remark}{Remark}

\renewcommand{\le}{\leqslant}
\renewcommand{\ge}{\geqslant}

\newcommand{\bC}[2]{C_{#1}^{#2}}   
\newcommand{\EE}{\mathsf{E}}
\newcommand{\PP}{\mathsf{P}}
\newcommand{\DD}{\mathsf{D}}

\begin{document}

\title[On the stability of the independence number]
      {On the Stability of the Independence Number of the Distance\\
       Graph $G(n,r,<s)$ for Growing $r$ and $s$}

\author{V.~A.~Pokhachevskiy}
\address{Lomonosov Moscow State University, Moscow, Russia}
\email{pokhachevskiy@gmail.com}

\author{A.~M.~Raigorodskii}
\address{Moscow Institute of Physics and Technology (National Research
University), Dolgoprudny, Moscow Region, Russia;
Lomonosov Moscow State University, Moscow, Russia;
Caucasus Mathematical Center, Adyghe State University, Maikop, Russia;
Buryat State University, Institute for Mathematics and Informatics,
Ulan-Ude, Russia}
\email{mraigor@yandex.ru}


\date{\today}

\begin{abstract}
We consider a random subgraph $G_p(n,r,<s)$ of the complete distance graph
$G(n,r,<s)$ whose vertices are the $r$-element subsets of the set
$\{1,\dots,n\}$ and whose edges join pairs of subsets that intersect in fewer
than $s$ elements; each edge survives independently of the others with
probability $p$. The stability of the Erd\H{o}s--Ko--Rado theorem is the event
that the independence number of the random subgraph equals that of the original
graph, $N=\bC{n-s}{r-s}$. For fixed $r$ and $s$, two-sided bounds for the
critical probability, whose ratio depends on $r$ and $s$, were obtained by
P.~A.~Ogarok and A.~M.~Raigorodskii. In this paper we establish the order of the
critical probability for growing $r=r(n)$ and $s=s(n)$: if
$\min\{\bC{r}{s}r^2,\,sr^3\}=o(n)$, then the upper and lower bounds differ only
by an absolute constant factor. The proof is elementary; only one of the two
branches of the growth condition uses the complete nontrivial-intersection
theorem of Ahlswede--Khachatrian.
\end{abstract}

\maketitle


\medskip\noindent A shortened Russian version of this paper, without
proofs, is published in \emph{Problemy Peredachi Informatsii} under the
title ``The order of the stability threshold for the Erd\H{o}s--Ko--Rado
theorem with growing parameters''~\cite{PRshort}.

\section{Introduction and statement of the result}
\label{sec:intro}

Let $n$, $r$, $s$ be positive integers, $s<r<n$, and
$[n]=\{1,\dots,n\}$. Consider the graph
$G(n,r,<s)=\bigl(V(n,r),E(n,r,s)\bigr)$, where
\begin{displaymath}
V(n,r)=\bigl\{x\subset[n]\colon|x|=r\bigr\},\qquad
E(n,r,s)=\bigl\{\{x,y\}\colon|x\cap y|<s\bigr\}.
\end{displaymath}
The graph $G(n,r,<s)$ is called a \emph{complete distance graph}; for
$s=1$ it is the Kneser graph. The independent sets of the graph
$G(n,r,<s)$ are exactly the \emph{$s$-intersecting} families of
$r$-element sets, so that the question of the independence number
$\alpha\bigl(G(n,r,<s)\bigr)$ is the classical
Erd\H{o}s--Ko--Rado problem \cite{EKR}. The simplest example of an
$s$-intersecting family is a \emph{star}
$S_W=\{x\in V(n,r)\colon x\supseteq W\}$ centered at $W\subset[n]$,
$|W|=s$; its cardinality equals $\bC{n-s}{r-s}$. By the
Frankl--Wilson theorem and the Ahlswede--Khachatrian complete
intersection theorem (see~\cite{Frankl,Wilson,AK}), for
$n\ge(r-s+1)(s+1)$
\begin{displaymath}
\alpha\bigl(G(n,r,<s)\bigr)=\bC{n-s}{r-s},
\end{displaymath}
and moreover for $n>(r-s+1)(s+1)$ the maximum is attained only on stars.

Let $p=p(n)\in[0,1]$. By $G_p(n,r,<s)$ we denote the random
subgraph of the graph $G(n,r,<s)$ in which each edge survives
independently of the others with probability $p$. Since the
independence number does not decrease when edges are deleted, one
always has $\alpha\bigl(G_p(n,r,<s)\bigr)\ge\bC{n-s}{r-s}$. The main
question is to find the critical probability upon crossing which the property
\begin{equation}
\label{eq:stability}
\alpha\bigl(G_p(n,r,<s)\bigr)=\bC{n-s}{r-s}
\end{equation}
(the \emph{stability} of the Erd\H{o}s--Ko--Rado theorem) is replaced
by its negation. The history of this question is presented in detail in
\cite{BGPR2,Pyad15,BNR,OgRai,Rai26,RaiSin}; we recall only the results
we need.

For the Kneser case $s=1$, Bollob\'as, Narayanan and Raigorodskii
\cite{BNR} found the sharp threshold: for $r=r(n)=o(n^{1/3})$, $r\ge2$,
the critical probability is
\begin{equation}
\label{eq:bnr}
p_c(n,r)=\frac{(r+1)\ln n-r\ln r}{\bC{n-1}{r-1}}
\end{equation}
(later, in \cite{DasTran,DevlinKahn,BKL}, the results for $s=1$ were
extended to the whole range $n\ge2r+1$). For arbitrary
\emph{fixed} $r$ and $s$, two-sided bounds for the threshold were
obtained by P.\,A.~Ogarok and A.\,M.~Raigorodskii \cite{OgRai}.

\begin{thmA}[Ogarok, Raigorodskii \cite{OgRai}]
\label{thOg}
Let $r>s\ge1$ be fixed positive integers, and let $\varepsilon>0$ be
fixed. Then for $p=2s\bC{r}{s}\ln n\big/\bC{n-s}{r-s}$ and under the
additional assumption $r>3$, asymptotically almost surely
\eqref{eq:stability} holds, and moreover all independent sets of
maximum cardinality in $G_p(n,r,<s)$ are stars; for
$p=(1-\varepsilon)(r+s)\ln n\big/\bC{n-s}{r-s}$, asymptotically almost
surely $\alpha\bigl(G_p(n,r,<s)\bigr)\ge\bC{n-s}{r-s}+1$.
\end{thmA}

For growing $r=r(n)$, $s=s(n)$, the known sufficient conditions for
stability have a completely different, polynomial in $n$, scale: in
\cite{KOR,RaiKaras} stability is proved, in essence, for
$p\gg r^3\ln n/n$ (for $r-s\ge2$; for $r-s=1$ more is required there),
and in \cite{PR}, for $r,s\to\infty$, $s=o(r)$, $r^2=o(n)$, for
$p\ge16sr^2\ln(n/r)/n$ (this condition is meaningful for
$sr^2=o(n/\ln n)$). Meanwhile, for fixed parameters the true
threshold, as Theorem~\ref{thOg} shows, has order
$\ln n\big/\bC{n-s}{r-s}\asymp n^{-(r-s)}\ln n$, i.e., for $r-s\ge2$ it
is incomparably smaller.

The aim of the present paper is to establish the \emph{order} of the
critical probability for growing parameters. Throughout what follows,
\begin{displaymath}
V:=\bC{n}{r},\qquad N:=\bC{n-s}{r-s},\qquad c:=\bC{r}{s},\qquad
\Lambda:=\ln\bigl(\bC{n}{s}(V-N)\bigr)
\end{displaymath}
(the symbol $V$ without arguments always denotes the number of
vertices $\bC{n}{r}$, whereas $V(n,r)$ is the vertex set itself). The
quantity $e^{\Lambda}$ is exactly the number of pairs $(W,A)$, where
$W$ is the center of a star and $A\notin S_W$ is a vertex; such pairs
give the simplest configurations destroying stability. Under the
hypotheses of Theorem~\ref{th:main} below,
\begin{displaymath}
\Lambda=(1+o(1))\Bigl(r\ln\frac nr+s\ln\frac ns\Bigr);
\end{displaymath}
this will be shown in Section~\ref{sec:prelim} (see \eqref{eq:Lam}
and~\eqref{eq:LamL}).

\begin{theorem}
\label{th:main}
Let $n\to\infty$, and let $r=r(n)$ and $s=s(n)$ be such that $r\ge2$,
$1\le s<r$ and
\begin{equation}
\label{eq:uslovie}
\min\bigl\{\bC{r}{s}\,r^2,\ s\,r^3\bigr\}=o(n).
\end{equation}
Then the following assertions hold.
\begin{itemize}
\item[{\rm(i)}] For
\begin{displaymath}
p\ \ge\ p_1(n):=\frac{64\,\Lambda}{N}
=\frac{64\,\ln\bigl(\bC{n}{s}(V-N)\bigr)}{\bC{n-s}{r-s}}
\end{displaymath}
asymptotically almost surely
$\alpha\bigl(G_p(n,r,<s)\bigr)=\bC{n-s}{r-s}$ and all independent
sets of maximum cardinality in the graph $G_p(n,r,<s)$ are stars.
\item[{\rm(ii)}] For any fixed $\varepsilon\in(0,1)$, for
\begin{displaymath}
p\ \le\ p_2(n):=
\frac{(1-\varepsilon)\Bigl(r\ln\dfrac{n}{r}+s\ln\dfrac{n}{s}\Bigr)}
{\bC{n-s}{r-s}}
\end{displaymath}
asymptotically almost surely
$\alpha\bigl(G_p(n,r,<s)\bigr)\ge\bC{n-s}{r-s}+1$.
\end{itemize}
\end{theorem}

\begin{corollary}
\label{sled:order}
Under the hypotheses of Theorem~\ref{th:main}, the critical
probability for stability has order
\begin{displaymath}
\frac{\Lambda}{N}\ \asymp\
\frac{r\ln\frac nr+s\ln\frac ns}{\bC{n-s}{r-s}}\,;
\end{displaymath}
more precisely, $p_1/p_2\le64/(1-\varepsilon)+o(1)$, i.e., the upper
and lower bounds differ only by an absolute constant factor.
\end{corollary}

Corollary~\ref{sled:order} follows from \eqref{eq:LamL}
(Section~\ref{sec:prelim}) and the monotonicity of the assertions of
the theorem in $p$ (Remark~\ref{zam:monot}).

In all remarks of this section, $n$ is assumed to be sufficiently
large and the hypotheses of Theorem~\ref{th:main} are assumed to
hold. The remarks rely on several elementary estimates, which are
derived in Section~\ref{sec:prelim} independently of them: the
inequality \eqref{eq:cN} for $c/N$, the inequality $N\ge n-s\ge n-r$
from \eqref{eq:Nnr}, the bounds \eqref{eq:Lam}--\eqref{eq:LamV} for
$\Lambda$, and the asymptotics $\Lambda=(1+o(1))L$ from
\eqref{eq:LamL}, where $L=r\ln\frac nr+s\ln\frac ns$. References to the
statements of Sections~\ref{sec:prelim}--\ref{sec:concl} are
everywhere given with the section indicated.

\begin{remark}
\label{zam:monot}
The assertions of Theorem~\ref{th:main} are monotone in $p$: for
$p'\ge p$ the random graph $G_{p'}$ can be realized as the union of
$G_p$ with an independent addition of edges, and every set independent
in $G_{p'}$ is independent in $G_p$. Therefore, both parts are stated
at once for half-intervals of values of $p$, and in the proofs it
suffices to consider $p=p_1$ in part (i) and $p=p_2$ in part (ii).
\end{remark}

\begin{remark}
\label{zam:deterministic}
Condition \eqref{eq:uslovie} implies $r^2=o(n)$, while
$(r-s+1)(s+1)\le r^2$ (since $s+1\le r$ and $r-s+1\le r$), so that
$n>(r-s+1)(s+1)$ for large $n$, and by the Ahlswede--Khachatrian
theorem \cite{AK} we indeed have
$\alpha\bigl(G(n,r,<s)\bigr)=\bC{n-s}{r-s}$. Thus, part (i) asserts
precisely stability. Note that the proof of Theorem~\ref{th:main}
itself uses neither the Frankl--Wilson theorem nor the complete
intersection theorem \cite{AK}; the only external result on which it
relies is the complete nontrivial-intersection theorem \cite{AK2}, and
it is needed only in the second branch of condition \eqref{eq:uslovie}
(Lemma~\ref{lem:HMF} in Section~\ref{sec:small}).
\end{remark}

\begin{remark}
\label{zam:s1}
For $s=1$, condition \eqref{eq:uslovie} turns into $r^3=o(n)$ --- the
same range $r=o(n^{1/3})$ as in the
Bollob\'as--Narayanan--Raigorodskii theorem \cite{BNR}. Here
$r\ln\frac nr+\ln n=(r+1)\ln n-r\ln r$, i.e., $p_2=(1-\varepsilon)
p_c(n,r)$: part (ii) reproduces the lower bound \eqref{eq:bnr} up to
$\varepsilon$. Of course, for $s=1$ only part (ii) is meaningful: in
this case the sharp threshold \eqref{eq:bnr} also gives the upper
bound. The coincidence of the ranges $r=o(n^{1/3})$ indirectly
indicates that condition \eqref{eq:uslovie} is a natural limit of
applicability of the present method.
\end{remark}

\begin{remark}
\label{zam:compare}
Let us compare part (i) with the known results on growing parameters,
in which the sufficient conditions for stability have the form
$p\gg r^3\ln n/n$ \cite{KOR,RaiKaras} and $p\ge16sr^2\ln(n/r)/n$
\cite{PR}, i.e., are polynomial in $n$. From \eqref{eq:Lam} and the
bound $N\ge n-r\ge n/2$ (see \eqref{eq:Nnr} in
Section~\ref{sec:prelim}) it follows that
\begin{equation}
\label{eq:p1poly}
p_1=\frac{64\Lambda}{N}\le\frac{256\,r(\ln n+1)}{n}.
\end{equation}
Therefore the ratio of $p_1$ to the bound $r^3\ln n/n$ does not exceed
$256(\ln n+1)/(r^2\ln n)\le512/r^2$, and to the bound
$16sr^2\ln(n/r)/n$ --- the quantity
$16(\ln n+1)\big/\bigl(s\,r\ln(n/r)\bigr)\le48/(sr)$ for large $n$
(here $\ln(n/r)\ge\frac12\ln n$, since $r=o(\sqrt n)$). Thus,
throughout the range \eqref{eq:uslovie}, part (i) is, up to the order of magnitude, at least as strong as
the bounds \cite{KOR,RaiKaras,PR}, and for $r\to\infty$
(respectively for $sr\to\infty$, which for \cite{PR} always holds,
since there $s\to\infty$) is stronger than them by at least a factor
of $\asymp r^2$ (respectively $\asymp sr$). These comparisons rely only
on the crude bound $N\ge n/2$; in reality, for bounded $r$ and $s$ the
gap is much larger: then $N\asymp n^{r-s}$ and $\Lambda\asymp\ln n$, so
that $p_1=64\Lambda/N\asymp n^{-(r-s)}\ln n$, whereas the polynomial
bounds have order $\ln n/n$; for $r-s\ge2$ this is a gap by a factor of
$\asymp n^{r-s-1}$.

For $r-s\to\infty$ the gain becomes superpolynomial. From
\eqref{eq:uslovie} it follows that $r^3=o(n)$, i.e.,
$\frac13\ln n-\ln r\to+\infty$, in particular
$\frac13\ln n-\ln r\ge1$ for large $n$; hence, in this case,
$\ln\frac{n-s}{r}\ge\ln n-1-\ln r\ge\frac23\ln n$, and, using
$c\ge r$, the bound $\Lambda\le2r(\ln n+1)$ and inequality
\eqref{eq:cN} (Section~\ref{sec:prelim}), we obtain
\begin{align*}
\ln\frac1{p_1}&=\ln N-\ln(64\Lambda)
\ge\ln\frac Nc+\ln r-\ln\bigl(128\,r(\ln n+1)\bigr)\\
&\ge(r-s)\ln\frac{n-s}{r}-\ln\ln n-6
\ \ge\ \frac23(r-s)\ln n-\ln\ln n-6 .
\end{align*}
For example, for any fixed $s$ condition \eqref{eq:uslovie} means
$r^3=o(n)$ (for $s\ge2$ and growing $r$ the minimum in
\eqref{eq:uslovie} is attained at the second expression, since
$\bC rs\,r^2\asymp r^{s+2}/s!\ge s\,r^3$ for large $r$); for
$r=\lfloor n^{1/4}\rfloor$ the stability threshold is lowered from
polynomial to the quantity $\exp\{-\Theta(n^{1/4}\ln n)\}$. Condition
\eqref{eq:uslovie} also allows joint growth of the parameters: since
$\bC{r}{s}\le(er/s)^s$, for the first of the two expressions in
\eqref{eq:uslovie} it suffices that
$s\ln\frac{er}{s}+2\ln r\le(1-\gamma)\ln n$ for some fixed $\gamma>0$;
for example, $r=\lfloor(\ln n)^a\rfloor$, $s=\lfloor(\ln n)^b\rfloor$
are suitable for any fixed $a>b\ge0$, $b<1$, and, thanks to the second
expression, so are $r=\lfloor n^{1/5}\rfloor$, $s=\lfloor r/2\rfloor$.

All of the above concerns the range \eqref{eq:uslovie}; outside it the
results \cite{KOR,RaiKaras,PR} remain the best known, and their ranges
of applicability are not contained in \eqref{eq:uslovie}. First, for
$s\to\infty$, $s=o(r)$ the minimum in \eqref{eq:uslovie} is attained at
$sr^3$ (for large $n$ one has $2\le s\le r-3$ and $r\ge6$, i.e.,
$\bC rs>sr$; see Remark~\ref{zam:uslovie} below), i.e., our condition
is $sr^3=o(n)$, whereas the bound \cite{PR} is already meaningful for
$sr^2=o(n/\ln n)$: for example, for $r=\lfloor n^{2/5}\rfloor$,
$s=\lfloor n^{1/20}\rfloor$ we have $sr^2\asymp n^{17/20}$ and \cite{PR}
gives stability for $p\asymp n^{-3/20}\ln n$, while $sr^3\asymp n^{5/4}$,
and Theorem~\ref{th:main} is inapplicable. Second, the bound
$p\gg r^3\ln n/n$ of the papers \cite{KOR,RaiKaras} does not depend on
$s$ for $r-s\ge3$: for $r=\lfloor n^{3/10}\rfloor$,
$s=\lfloor r/2\rfloor$ they give stability for $p\gg n^{-1/10}\ln n$,
whereas $sr^3\asymp n^{6/5}$ and $\bC rs\ge2^r/(r+1)$, so that
\eqref{eq:uslovie} is violated. For \emph{fixed} $s$, on the contrary,
the condition $r^3\ln n=o(n)$ implies \eqref{eq:uslovie}, i.e., the
range of \cite{KOR,RaiKaras} lies entirely within ours; the gain here
can be large even for bounded parameters: say, for fixed $r$ and
$s=r-1$ we have $N=n-r+1$ and $\Lambda\asymp\ln n$, i.e.,
$p_1\asymp\ln n/n$, whereas in \cite{RaiKaras} for $r-s=1$ one requires
$p\gg\ln n/n^{1/3}$.
Finally, the comparison with Theorem~\ref{thOg} for small $s\bC rs$ is
treated in Remark~\ref{zam:constants}.
\end{remark}

\begin{remark}
\label{zam:constants}
The absolute constant 64 in part (i) is not optimal; it arises as
$\max\{4/\beta,2/\eta\}$, where $\beta=\eta=\frac1{16}$ are the
constants from Corollary~\ref{sled:Lstar} (Section~\ref{sec:small}),
and any improvement of them proportionally improves the bound (see
Section~\ref{sec:concl}). Comparison with Theorem~\ref{thOg}: the
latter corresponds to the value $2s\bC rs$ where we have
$64\Lambda/\ln n\le64(r+s)$ (since
$\bC ns(V-N)\le\bC ns\bC nr\le n^{s+r}$), so that for small $s\bC rs$
(for example, for $s=1$, where $s\bC rs=r$) the constant of
Theorem~\ref{thOg} is smaller than ours, while for
$s\bC rs\gtrsim32(r+s)$ (for example, for $r=10$, $s=5$:
$s\bC rs=1260$ against $32(r+s)=480$) it is larger. The advantage of
part (i) is that its constant is \emph{absolute}: the ratio of the
upper and lower bounds in Theorem~\ref{thOg} grows as
$s\bC{r}{s}/(r+s)$, whereas in Theorem~\ref{th:main} it is bounded.
\end{remark}

\begin{remark}
\label{zam:uslovie}
Condition \eqref{eq:uslovie} is the minimum of two expressions; the
second is better than the first exactly when $\bC{r}{s}>sr$, i.e., for
$2\le s\le r-3$ and $r\ge6$. Thus, for $s=2$ the first expression has
order $r^4$, the second $r^3$; for $s\asymp r/2$ the first grows
exponentially in $r$, the second as $r^4$. For $s\in\{r-2,r-1\}$ the
first expression is no worse than the second, and strictly better for
all such $(r,s)$ except $(2,1)$ and $(3,1)$, where the expressions
coincide. The first of the two conditions is ensured elementarily
(Lemma~\ref{lem:nontriv} in Section~\ref{sec:small}), the second uses
the complete nontrivial-intersection theorem \cite{AK2}
(Lemma~\ref{lem:HMF} in the same place) --- this is the only place in
the paper where an external result is needed.
\end{remark}

\begin{remark}
\label{zam:iii}
Part (ii) holds also under conditions weaker than \eqref{eq:uslovie},
namely $r\ge2$, $1\le s<r$, $r^2=o(n)$; this is clear from the proof
(Section~\ref{sec:lower}), in which condition \eqref{eq:uslovie} is not
used. The condition $r^2=o(n)$ follows from \eqref{eq:uslovie}
(Remark~\ref{zam:deterministic}).
\end{remark}

The rest of the paper is organized as follows. In
Section~\ref{sec:prelim} we introduce the notation and prove two
simple lemmas. Section~\ref{sec:small} is the main part of the paper:
there, from a single estimate (Lemma~\ref{lem:cert}) it is deduced
that every $N$-element family of small star degree contains
$\gtrsim N^2$ edges (Corollary~\ref{sled:Lstar}). In
Section~\ref{sec:upper} part (i) is proved, in Section~\ref{sec:lower}
part (ii); Section~\ref{sec:concl} contains concluding remarks.

The general scheme of the proof of part (i) follows \cite{OgRai}
(which, in turn, follows \cite{BNR}). Let us specify what is new in
Section~\ref{sec:small} and what is borrowed. New is, first of all,
Lemma~\ref{lem:Lstar}: a family of cardinality $M$ and small star
degree contains at least $M^2/16$ edges provided that two quantities
depending only on $n$, $r$, $s$ --- the maximum number $\Phi$ of common
non-neighbors of the endpoints of an edge and the maximum cardinality
$\Psi$ of a nontrivial $s$-intersecting family --- are small compared
with $M$. It is this estimate, uniform in $r$ and $s$, that yields an
absolute constant in part (i). Also new are the bounds on $\Phi$ and
$\Psi$ for growing $r$ and $s$
(Lemmas~\ref{lem:pair}--\ref{lem:cert2}; together with
Lemma~\ref{lem:HMF} they give the two branches of condition
\eqref{eq:uslovie}); the refined bound on $\Phi$ in
Lemma~\ref{lem:cert2} is obtained by splitting a common non-neighbor
into the parts lying in the intersection and in the differences of the
endpoints of the edge. The following techniques are borrowed. The
quantity $\Phi$ itself and the idea of comparing it with the
cardinality of an independent set go back to \cite{BGPR1,BGPR2,Rai17} (see
Remark~\ref{zam:Phi} in Section~\ref{sec:small}). The certificate
lemma (Lemma~\ref{lem:cert}) is based on the same choice of an
additional element as the proof of Lemma~2\,(ii) of our paper
\cite{PR}; only its application to bounding $\Phi$ and $\Psi$ is new.
Finally, the bound on $\Psi$ in Lemma~\ref{lem:HMF} is a direct
consequence of the Ahlswede--Khachatrian theorem \cite{AK2}, analogous
to Lemma~1\,(d) of \cite{PR}, but without the assumption $s=o(r)$. In
addition, the choice of ranges in Section~\ref{sec:upper} and the
explicit tracking of all error terms, which for fixed parameters were
constants, are new.

\section{Notation and auxiliary results}
\label{sec:prelim}

Throughout what follows $n\to\infty$; all asymptotics refer to this
limit, all statements ``for all sufficiently large $n$'' are used
without stipulation, and all $o(1)$ that appear are uniform in all the
remaining variables (indices of summands, ranges of indices, etc.). By
$G$ we denote the graph $G(n,r,<s)$, and by $G_p$ its random subgraph.
Put
\begin{displaymath}
\delta:=\frac{2r^2}{n}.
\end{displaymath}
In Sections~\ref{sec:small}--\ref{sec:lower} we assume $r^2=o(n)$
(and in Section~\ref{sec:upper} also \eqref{eq:uslovie}); in
particular, $\delta\to0$ and $r=o(\sqrt n)$.

For $D\subseteq V(n,r)$ we denote by $e(D)$ the number of edges of the
graph $G$ inside $D$, and by
\begin{displaymath}
\deg D:=\max_{W\subset[n],\,|W|=s}\bigl|D\cap S_W\bigr|
\end{displaymath}
the \emph{star degree} of the family $D$ (following \cite{OgRai});
clearly $1\le\deg D\le\min(|D|,N)$ for nonempty $D$. For $x\in D$ we
put
\begin{displaymath}
D_x:=\bigl\{y\in D\setminus\{x\}\colon|x\cap y|\ge s\bigr\},
\qquad \nu(x):=|D_x|
\end{displaymath}
the set and the number of \emph{non-neighbors} of the vertex $x$
inside $D$. Since each pair of distinct vertices of $D$ is either
joined by an edge of the graph $G$ or not,
\begin{equation}
\label{eq:nu}
e(D)=\bC{|D|}{2}-\frac12\sum_{x\in D}\nu(x).
\end{equation}
A family $B\subseteq V(n,r)$ is called \emph{$s$-intersecting} if
$|x\cap y|\ge s$ for all $x,y\in B$, and \emph{trivial} if there
exists an $s$-set $W$ with $B\subseteq S_W$.

We shall need the standard estimates: for positive integers
$a\ge b\ge1$
\begin{equation}
\label{eq:binom}
\Bigl(\frac{a}{b}\Bigr)^b\le \bC{a}{b}\le\Bigl(\frac{ea}{b}\Bigr)^{b},
\end{equation}
as well as $1+x\le e^x$ for all $x$ and, for $0\le p\le1/2$, the
inequalities
\begin{equation}
\label{eq:pp}
(1-p)^{-1}=1+\frac{p}{1-p}\le1+2p\le e^{2p},\qquad
1-p\ge e^{-p-p^2}.
\end{equation}
We shall repeatedly need the following elementary observation: for
$x<b$ the function $x\mapsto\frac{a-x}{b-x}$ \emph{decreases} if
$a<b$ and \emph{increases} if $a>b$ (its derivative equals
$\frac{a-b}{(b-x)^2}$). It is precisely this that justifies all steps
of the form ``a fraction is estimated by its value at a different
shift'' --- in \eqref{eq:cN} and \eqref{eq:shift} below, in
Lemma~\ref{lem:Mt} and in Lemmas~\ref{lem:cert2} and~\ref{lem:HMF}; in
what follows we use it without further reference.

Using the symmetry $c=\bC{r}{s}=\bC{r}{r-s}$, we obtain
\begin{equation}
\label{eq:cN}
\frac{c}{N}=\frac{\bC{r}{r-s}}{\bC{n-s}{r-s}}
=\prod_{i=0}^{r-s-1}\frac{r-i}{n-s-i}
\le\Bigl(\frac{r}{n-s}\Bigr)^{r-s}\le\frac{r}{n-s},
\end{equation}
since $\frac{r-i}{n-s-i}\le\frac{r}{n-s}$ for $0\le i<r-s$, provided
$r\le n-s$ (the latter, for $r^2=o(n)$, holds for large $n$; it is
only under this condition that \eqref{eq:cN} is used). Note also that
for $1\le s<r<n$
\begin{equation}
\label{eq:Nnr}
N=\bC{n-s}{r-s}\ge n-s\ge n-r
\end{equation}
(since $1\le r-s\le n-s-1$), in particular $N\to\infty$ and $N\ge n/2$
for $n\ge2r$. Finally, throughout we use the elementary inequality
\begin{equation}
\label{eq:shift}
\bC{n-s-1}{r-s-1}=N\,\frac{r-s}{n-s}\le\frac{N\,r}{n},
\end{equation}
valid for $r\le n$ (indeed, $n(r-s)\le r(n-s)$ is equivalent to
$rs\le ns$).

\begin{lemma}
\label{lem:identity}
$V\,\bC{r}{s}=\bC{n}{s}\,N$, i.e., $\dfrac{V}{N}=\dfrac{\bC{n}{s}}{c}$.
\end{lemma}

\begin{proof}
Both sides of the identity $\bC{n}{r}\bC{r}{s}=\bC{n}{s}\bC{n-s}{r-s}$
count the number of pairs $(x,W)$, where $x\in V(n,r)$ and
$W\subseteq x$, $|W|=s$: on the left one first chooses $x$, then $W$;
on the right the other way round.
\end{proof}

From Lemma~\ref{lem:identity} and \eqref{eq:binom} it follows that
$N/V=c\big/\bC{n}{s}\le(er/n)^s\to0$, so that for large $n$ we have
$V-N\ge V/2$ and hence
\begin{equation}
\label{eq:Lam}
s\ln\frac ns+r\ln\frac nr-1\ \le\ \Lambda\ \le\
s\ln\frac{en}{s}+r\ln\frac{en}{r}\ \le\ 2r(\ln n+1),
\end{equation}
\begin{equation}
\label{eq:LamV}
\ln V\le\Lambda+\ln2,\qquad
\ln\frac VN\le\Lambda+\ln2-\ln N,\qquad
\ln\bC ns\le\Lambda .
\end{equation}
Put $L:=r\ln\frac nr+s\ln\frac ns$. Then the right-hand side of
\eqref{eq:Lam} equals $L+r+s$, while $L\ge r\ln\frac nr\ge\frac r2\ln n$
for large $n$ (since $r=o(\sqrt n)$), so that
$(r+s)/L\le2/\ln\frac nr\to0$. Therefore
\begin{equation}
\label{eq:LamL}
\Lambda=(1+o(1))\,L,\qquad \Lambda\ge r\ln\frac nr\to\infty,
\end{equation}

Let $W\subset[n]$, $|W|=s$, and let a vertex $A\in V(n,r)$ intersect
$W$ in exactly $t\le s-1$ elements. By $M_t$ we denote the number of
vertices of the star $S_W$ joined to $A$ by an edge of the graph $G$
(it is easily seen that it depends only on $t$).

\begin{lemma}
\label{lem:Mt}
For all $0\le t\le s-1$ and sufficiently large $n$
\begin{displaymath}
(1-\delta)N\le M_t\le N.
\end{displaymath}
\end{lemma}

\begin{proof}
The inequality $M_t\le|S_W|=N$ is obvious. Further, every vertex of
the form $w=W\sqcup T$, where $T\subset[n]\setminus(W\cup A)$,
$|T|=r-s$, belongs to $S_W$ and satisfies $|w\cap A|=|W\cap A|=t<s$,
i.e., is joined to $A$ by an edge. The number of such vertices equals
$\bC{n-s-(r-t)}{r-s}$, so that
\begin{displaymath}
\frac{M_t}{N}\ge\frac{\bC{n-s-r+t}{r-s}}{\bC{n-s}{r-s}}
=\prod_{i=0}^{r-s-1}\frac{n-s-r+t-i}{n-s-i}
\ge\Bigl(1-\frac{r-t}{n-r}\Bigr)^{r-s}
\ge1-\frac{r(r-s)}{n-r}\ge1-\frac{2r^2}{n}
\end{displaymath}
for large $n$ (in the penultimate step Bernoulli's inequality was
used).
\end{proof}

Finally, we adopt the following convention: if a set $D$ is independent in
$G_p$, then none of the $e(D)$ edges of the graph $G$ inside $D$ has
survived, and the probability of this event equals
$(1-p)^{e(D)}\le e^{-p\,e(D)}$. We sum estimates of this kind over all
sets of the type under consideration, obtaining a bound on the
expectation of their number, and use Markov's inequality
$\PP(X>0)\le\EE(X)$ for integer-valued nonnegative $X$; this is done
repeatedly without further stipulation.

\section{Families of small star degree}
\label{sec:small}

In this section we prove the main new estimate: a family of
cardinality $N$ that is not concentrated near any single star contains
at least $N^2/16$ edges. All statements of the section are
deterministic.

Let us introduce two quantities depending only on $n$, $r$, $s$:
\begin{align}
\label{eq:Phi}
\Phi&:=\max_{\substack{x,x'\in V(n,r)\\ |x\cap x'|<s}}
\bigl|\{y\in V(n,r)\colon|y\cap x|\ge s,\ |y\cap x'|\ge s\}\bigr|,\\
\label{eq:Psi}
\Psi&:=\max\bigl\{|B|\colon B\subseteq V(n,r)\ \text{is a nontrivial
$s$-intersecting family}\bigr\}
\end{align}
(if there are no pairs $x,x'$ with $|x\cap x'|<s$, we set $\Phi:=0$;
if there are no nontrivial $s$-intersecting families, we set
$\Psi:=1$). The quantity $\Phi$ is the maximum number of ``common
non-neighbors'' of two vertices joined by an edge, and $\Psi$ is the
maximum cardinality of an independent family not lying entirely in any
star.

\begin{remark}
\label{zam:Phi}
The quantity $\Phi$ --- the maximum over edges of the graph $G$ of the
number of vertices adjacent to neither endpoint of the edge --- is not
new: for $s=1$ and $p=\frac12$ it (under the notation $\gamma_n$) was
considered by A.\,M.~Raigorodskii \cite{Rai17}, see also Theorem~17 in
Section~3.3.10 of the survey \cite{RaiSin}; related questions on the
independence numbers of random subgraphs were also considered in
\cite{Rai26}. The technique itself goes back to a lemma of
M.\,M.~Pyaderkin to the effect that two adjacent vertices have few
common non-neighbors \cite{BGPR1,BGPR2}. The difference of our
application is that $\Phi$ is compared not with $\alpha(G)$, but with
the cardinality of an arbitrary family of small star degree
(Lemma~\ref{lem:Lstar}), and that the bounds for $\Phi$ are derived
from the certificate lemma (Lemmas~\ref{lem:pair} and~\ref{lem:cert2})
and are therefore applicable for growing $r$ and $s$.
\end{remark}

\subsection{The main lemma}

\begin{lemma}
\label{lem:Lstar}
Let $1\le s<r$, $D\subseteq V(n,r)$, $M:=|D|$, and let
\begin{displaymath}
M\ge16,\qquad \deg D\le\frac{M}{16},\qquad
\Phi<\frac M2,\qquad \Psi\le\frac M{16}.
\end{displaymath}
Then $e(D)\ge\dfrac{M^2}{16}$.
\end{lemma}

\begin{proof}
Call a vertex $x\in D$ \emph{bad} if $\nu(x)\ge\frac34M$, and denote
by $B$ the set of bad vertices.

\emph{Step 1: $B$ is an $s$-intersecting family.} Let
$x,x'\in B$, $x\ne x'$. Since $D_x,D_{x'}\subseteq D$ and $|D|=M$,
\begin{displaymath}
|D_x\cap D_{x'}|\ \ge\ \nu(x)+\nu(x')-M\ \ge\
\frac34M+\frac34M-M=\frac M2 .
\end{displaymath}
If we had $|x\cap x'|<s$, then every vertex $y\in D_x\cap D_{x'}$
would satisfy $|y\cap x|\ge s$ and $|y\cap x'|\ge s$, i.e., would be
counted in \eqref{eq:Phi}; then $|D_x\cap D_{x'}|\le\Phi<M/2$ --- a
contradiction. Hence $|x\cap x'|\ge s$ for all distinct
$x,x'\in B$.

\emph{Step 2: $|B|\le\frac M{16}$.} For $|B|\le1$ this follows from
$M\ge16$. Let $|B|\ge2$. If the family $B$ is trivial, i.e.,
$B\subseteq S_W$ for some $s$-set $W$, then, since $B\subseteq D$,
\begin{displaymath}
|B|\le|D\cap S_W|\le\deg D\le\frac M{16}.
\end{displaymath}
If, however, $B$ is nontrivial, then by Step~1 and the definition
\eqref{eq:Psi} we have $|B|\le\Psi\le\frac M{16}$.

\emph{Step 3: counting.} For $x\notin B$ we have $\nu(x)<\frac34M$,
and always $\nu(x)\le M-1$. Therefore by Step~2
\begin{displaymath}
\sum_{x\in D}\nu(x)\ \le\ |B|\,(M-1)+\bigl(M-|B|\bigr)\frac34M
\ \le\ \frac{M^2}{16}+\frac34M^2=\frac{13}{16}M^2 ,
\end{displaymath}
and from \eqref{eq:nu}
\begin{displaymath}
e(D)\ \ge\ \frac{M(M-1)}{2}-\frac{13}{32}M^2
=\frac{3}{32}M^2-\frac M2\ \ge\ \frac{M^2}{16},
\end{displaymath}
where the last inequality is equivalent to $M^2/32\ge M/2$, i.e.,
$M\ge16$.
\end{proof}

It remains to estimate $\Phi$ and $\Psi$. We note at once that the
quantity $\Psi$ is known exactly: for $s=1$ this is the classical
Hilton--Milner theorem \cite{HM}, for fixed $r$ and $s$ and
sufficiently large $n$ it is a result of Frankl \cite{Frankl}, and in
the whole range $n>(s+1)(r-s+1)$ it is the Ahlswede--Khachatrian
complete nontrivial-intersection theorem \cite{AK2}. It is precisely
the latter that yields Lemma~\ref{lem:HMF} and the second branch of
condition \eqref{eq:uslovie}. However, we first derive cruder bounds
for both quantities (Lemmas~\ref{lem:cert}--\ref{lem:nontriv}): they
suffice for the first branch of \eqref{eq:uslovie}, and this branch
thus remains completely self-contained. Both the crude and the refined
bounds use one and the same elementary observation.

\begin{lemma}[Certificate lemma]
\label{lem:cert}
Let $T\subset[n]$, $|T|=s$, and $z\in V(n,r)$, where
$T\not\subseteq z$. Then
\begin{displaymath}
\bigl|\{y\in V(n,r)\colon T\subseteq y,\ |y\cap z|\ge s\}\bigr|
\ \le\ r\,\bC{n-s-1}{r-s-1}.
\end{displaymath}
\end{lemma}

\begin{proof}
Since $T\not\subseteq z$, we have $|T\cap z|\le s-1$. Let $y$ be
counted on the left. If we had $y\cap z\subseteq T$, then
$|y\cap z|\le|T\cap z|\le s-1<s$ --- a contradiction. Hence there
exists $u\in(y\cap z)\setminus T$, and then $y\supseteq T\cup\{u\}$,
where $u\in z\setminus T$, i.e., at most $|z|=r$ possibilities. The
number of vertices containing a fixed $(s+1)$-set equals
$\bC{n-s-1}{r-s-1}$.
\end{proof}

This choice of an additional element $u$ was already used in the proof
of Lemma~2\,(ii) of \cite{PR}, where it bounded the number of vertices
of a star that are not adjacent to a vertex outside it. Here the same
count is applied to bound the quantities $\Phi$ and $\Psi$.

Put
\begin{equation}
\label{eq:kappa}
\varkappa:=c\,r\,\bC{n-s-1}{r-s-1}
\ \overset{\eqref{eq:shift}}{\le}\ \frac{c\,r^2}{n}\,N .
\end{equation}

\begin{lemma}
\label{lem:pair}
$\Phi\le\varkappa$.
\end{lemma}

\begin{proof}
Let $|x\cap x'|<s$ and let a vertex $y$ be counted in \eqref{eq:Phi}.
From $|y\cap x|\ge s$ it follows that $y$ contains at least one
$s$-subset $T\subseteq x$; there are at most $\bC rs=c$ such $T$.
Moreover $T\not\subseteq x'$: otherwise $T\subseteq x\cap x'$ and
$|x\cap x'|\ge s$. It remains to apply Lemma~\ref{lem:cert} with
$z:=x'$ and sum over $T$.
\end{proof}

\begin{lemma}
\label{lem:nontriv}
$\Psi\le1+\varkappa$.
\end{lemma}

\begin{proof}
Let $B$ be a nontrivial $s$-intersecting family; we may assume
$|B|\ge2$. Fix $x_0\in B$. Each $y\in B\setminus\{x_0\}$ satisfies
$|y\cap x_0|\ge s$, i.e., contains at least one $s$-subset
$T\subseteq x_0$ (at most $c$ choices). Fix such a $T$. By
nontriviality there exists $z_T\in B$ with $T\not\subseteq z_T$. If
$y\in B$ and $T\subseteq y$, then $y\ne z_T$ and hence
$|y\cap z_T|\ge s$; by Lemma~\ref{lem:cert} (with $z:=z_T$) there are
at most $r\bC{n-s-1}{r-s-1}$ such $y$. Summing over $T\subseteq x_0$
and adding $x_0$ itself, we obtain $|B|\le1+c\,r\,\bC{n-s-1}{r-s-1}$.
\end{proof}

\subsection{Refined estimates}

Lemmas~\ref{lem:pair} and~\ref{lem:nontriv} give $\Phi,\Psi\lesssim
c\,r^2N/n$ and therefore lead to the condition $cr^2=o(n)$. Both
bounds can be improved by replacing the factor $c$ with factors
growing polynomially in $r$; this is what yields the second expression
in \eqref{eq:uslovie}. Here and in what follows $\bC ab:=0$ for $b>a$
or $b<0$.

\begin{lemma}
\label{lem:cert2}
Let $n\ge15\,s\,r^3$. Then $\Phi\le\dfrac{2e^2\,r^3}{n}\,N$.
\end{lemma}

\begin{proof}
Let $|x\cap x'|<s$, $m:=|x\cap x'|\le s-1$, and let a vertex $y$ be
counted in \eqref{eq:Phi}. Put $a:=|y\cap x\cap x'|\le m$ and
$j:=s-a\ge s-m\ge1$. From $|y\cap x|\ge s$ and
$y\cap x=(y\cap x\cap x')\sqcup\bigl(y\cap(x\setminus x')\bigr)$
it follows that $|y\cap(x\setminus x')|\ge j$; similarly
$|y\cap(x'\setminus x)|\ge j$. Hence $y$ contains a set
$A\sqcup P\sqcup P'$, where $A=y\cap x\cap x'$ ($|A|=s-j$),
$P\subseteq x\setminus x'$, $P'\subseteq x'\setminus x$,
$|P|=|P'|=j$; its cardinality equals $s+j$. The number of such sets
for fixed $j$ does not exceed
$\bC{m}{s-j}\bigl(\bC{r}{j}\bigr)^2
\le\bC{s-1}{j-1}\bigl(\bC{r}{j}\bigr)^2$, and the number of vertices
containing a fixed $(s+j)$-set equals $\bC{n-s-j}{r-s-j}$. Further,
\begin{displaymath}
\frac{\bC{n-s-j}{r-s-j}}{\bC{n-s}{r-s}}
=\prod_{i=0}^{j-1}\frac{r-s-i}{n-s-i}
\le\Bigl(\frac{r-s}{n-s}\Bigr)^{j}
\le\Bigl(\frac{r}{n}\Bigr)^{j},
\end{displaymath}
while $\bC{s-1}{j-1}\le s^{j-1}$ and $\bC{r}{j}\le(er/j)^j$ by
\eqref{eq:binom}. Therefore
\begin{displaymath}
\Phi\le\sum_{j\ge1}s^{j-1}\Bigl(\frac{er}{j}\Bigr)^{2j}
N\Bigl(\frac{r}{n}\Bigr)^{j}
=\frac Ns\sum_{j\ge1}\Bigl(\frac{s\,e^2r^3}{j^2\,n}\Bigr)^{j}
\le\frac Ns\sum_{j\ge1}q^{j},\qquad
q:=\frac{s\,e^2r^3}{n}.
\end{displaymath}
For $n\ge15sr^3\ge2e^2sr^3$ we have $q\le1/2$, so that
$\Phi\le\frac Ns\cdot2q=\frac{2e^2r^3}{n}N$.
\end{proof}

\begin{lemma}
\label{lem:HMF}
Let $1\le s<r$ and $n\ge2r^2$. Then $\Psi\le\delta N$, where
$\delta=\dfrac{2r^2}{n}$ is the quantity from Section~\ref{sec:prelim}.
\end{lemma}

\begin{proof}
Put
\begin{gather*}
\mathcal A:=\bigl\{y\in V(n,r)\colon|y\cap[s+2]|\ge s+1\bigr\},\\
\mathcal H:=\bigl\{y\in V(n,r)\colon[s]\subseteq y,\
y\cap\{s+1,\dots,r+1\}\ne\varnothing\bigr\}
\cup\bigl\{[r+1]\setminus\{i\}\colon i\in[s]\bigr\}.
\end{gather*}
(in the notation of \cite{AK2} these are the families $V_1$ and $V_2$.)
By the complete nontrivial-intersection theorem
\cite[Theorem, parts~(b),~(c)]{AK2} for $n>(s+1)(r-s+1)$ the
cardinality of any nontrivial $s$-intersecting family of $r$-subsets
of $[n]$ does not exceed $\max\{|\mathcal A|,|\mathcal H|\}$. The
condition is satisfied: $s+1\le r$ (since $s<r$) and $r-s+1\le r$
(since $s\ge1$), so that $(s+1)(r-s+1)\le r^2<2r^2\le n$.

Let us estimate both families using \eqref{eq:shift}. First, splitting
$\mathcal A$ according to the cardinality
$|y\cap[s+2]|\in\{s+1,s+2\}$ and using Pascal's identity
$\bC{n-s-2}{r-s-1}+\bC{n-s-2}{r-s-2}=\bC{n-s-1}{r-s-1}$,
\begin{displaymath}
|\mathcal A|=(s+2)\bC{n-s-2}{r-s-1}+\bC{n-s-2}{r-s-2}
\le(s+2)\bC{n-s-1}{r-s-1}\le\frac{(r+1)\,r}{n}\,N
\le\frac{2r^2}{n}\,N,
\end{displaymath}
since $s+2\le r+1\le2r$. Second,
$|\mathcal H|=N-\bC{n-r-1}{r-s}+s$, and by Bernoulli's inequality
\begin{multline*}
\frac{\bC{n-r-1}{r-s}}{N}=\prod_{i=0}^{r-s-1}\frac{n-r-1-i}{n-s-i}
\ge\Bigl(1-\frac{r-s+1}{n-r+1}\Bigr)^{r-s}\\
\ge1-\frac{(r-s)(r-s+1)}{n-r+1}\ge1-\frac{2r(r-1)}{n},
\end{multline*}
where we have used $(r-s)(r-s+1)\le r(r-1)$ (since $s\ge1$) and
$n-r+1\ge n/2$ (valid for $n\ge2r$). Hence
\begin{displaymath}
|\mathcal H|\le\frac{2r(r-1)}{n}N+s\le\frac{2r(r-1)}{n}N+\frac{2r}{n}N
=\frac{2r^2}{n}\,N,
\end{displaymath}
since $s\le r\le\frac{2rN}{n}$ as $N\ge n-r\ge n/2$
(see \eqref{eq:Nnr}).
\end{proof}

\subsection{Summary of the section}

\begin{corollary}
\label{sled:Lstar}
Let $1\le s<r$, and suppose that at least one of the following two
conditions holds:
\begin{displaymath}
\text{\rm(A)}\quad n\ge32\,c\,r^2,
\qquad\text{or}\qquad
\text{\rm(B)}\quad n\ge30\,s\,r^3 .
\end{displaymath}
Then for any $D\subseteq V(n,r)$ with $|D|=N$ and $\deg D\le N/16$
one has $e(D)\ge N^2/16$.
\end{corollary}

\begin{proof}
It suffices to verify the conditions of Lemma~\ref{lem:Lstar} for
$M=N$, i.e., the inequalities $\Phi<N/2$ and $\Psi\le N/16$. The
condition $M\ge16$ is satisfied in both cases: in case (B), from
$s\ge1$ it follows that $n\ge30r^3$, and in case (A) the same follows
from $c=\bC rs\ge r$ (since $1\le s\le r-1$), since then
$n\ge32cr^2\ge32r^3$. Hence $n\ge30\cdot2^3=240$ (since $r\ge2$) and
$r\le r^3\le n/30$, so that by \eqref{eq:Nnr}
\begin{displaymath}
M=N\ \ge\ n-r\ \ge\ \tfrac{29}{30}\,n\ \ge\ 232 ;
\end{displaymath}
in particular $N\ge32\ge16$.

In case (A), by \eqref{eq:kappa} and condition (A) we have
$\varkappa\le c\,r^2N/n\le N/32$, so that
Lemmas~\ref{lem:pair} and~\ref{lem:nontriv} give $\Phi\le N/32<N/2$ and
$\Psi\le1+N/32\le N/16$ (the last inequality, since $N\ge32$).

In case (B), from $n\ge30sr^3\ge15sr^3$ and Lemma~\ref{lem:cert2} it
follows that
\begin{displaymath}
\Phi\le\frac{2e^2r^3}{n}N\le\frac{2e^2}{30}\,N<\frac N2,
\end{displaymath}
since $n\ge30r^3$ and $2e^2<15$. Moreover,
$n\ge30sr^3\ge30r^3\ge32r^2$ for $r\ge2$, so that
Lemma~\ref{lem:HMF} is applicable (its condition $n\ge2r^2$ is
satisfied) and $\Psi\le2r^2N/n\le N/16$.

\end{proof}

\begin{remark}
\label{zam:whereuslovie}
Condition \eqref{eq:uslovie} is used in the paper \emph{only} in
Corollary~\ref{sled:Lstar}, and moreover in the following form: from
\eqref{eq:uslovie} it follows that for all sufficiently large $n$ one
has $\min\{32\bC rs\,r^2,\,30sr^3\}\le n$, i.e., for each such $n$
either (A) or (B) holds (which one exactly is unimportant and may
depend on $n$). All the other arguments of the proof of part (i)
require only $r^2=o(n)$; in part (ii) the condition $r^2=o(n)$
suffices as well (Remark~\ref{zam:iii}).
\end{remark}

\section{Proof of part (i)}
\label{sec:upper}

By monotonicity (Remark~\ref{zam:monot}) it suffices to prove the
assertion for $p=p_1=64\Lambda/N$. First of all, $p\to0$: by
\eqref{eq:Lam} and \eqref{eq:Nnr} (whence $N\ge n-r\ge n/2$)
\begin{displaymath}
p_1\le\frac{256\,r(\ln n+1)}{n}\to0,
\end{displaymath}
since $r^2=o(n)$ implies $r=o(\sqrt n)$ and hence $r\ln n=o(n)$.

Following \cite{OgRai}, let us introduce random variables. For
$i\in\mathbb N$ let $X_i$ be the number of sets $D$ independent in
$G_p$ of cardinality $N$ with degree $\deg D=N-i$. Let $Y$ be the
number of sets independent in $G_p$ of cardinality $N+1$ containing a
whole star. If $Y=0$ and $X_i=0$ for all $i\ge1$, then, first, every
independent set of cardinality $N$ has degree $N$, i.e., contains $N$
vertices of some star and hence coincides with it; second, there are
no independent sets of cardinality $N+1$: such a set $F$ would contain
an independent subset of cardinality $N$, i.e., (by what has just been
proved) a star, and then $F$ would be counted by the quantity $Y$,
contrary to $Y=0$; finally, every independent set of cardinality
greater than $N+1$ contains an independent subset of cardinality
$N+1$. Since $\alpha(G_p)\ge N$ (a star is independent in $G$, and a
fortiori in $G_p$), in this case $\alpha(G_p)=N$ and all independent
sets of maximum cardinality are stars.

Thus, it suffices to prove that
$\PP\bigl(Y>0\text{ or }\exists\,i\ge1\colon X_i>0\bigr)\to0$.
Fix $\varepsilon_0:=1/10$ and split the range of values of $i$ into
three parts:
\begin{displaymath}
\text{(1)}\ 1\le i\le\varepsilon_0N,\qquad
\text{(2)}\ \varepsilon_0N<i\le\tfrac{15}{16}N,\qquad
\text{(3)}\ i>\tfrac{15}{16}N .
\end{displaymath}
Note that case (3) corresponds exactly to the sets with degree
$\deg D=N-i<N/16$, i.e., those to which Corollary~\ref{sled:Lstar} is
applicable.

\subsection{The quantity $Y$}
\label{sec:Y}

A set counted by the quantity $Y$ is $S_W\cup\{A\}$, where $W$ is the
center of a star and $A\notin S_W$ is a vertex intersecting $W$ in
$t\le s-1$ elements; the number of such pairs $(W,A)$ equals
$\bC ns(V-N)=e^{\Lambda}$. By Lemma~\ref{lem:Mt} the vertex $A$ is
joined in $G$ to at least $(1-\delta)N$ vertices of the star $S_W$, and
all these edges must be absent in $G_p$. Therefore for large $n$ (when
$\delta\le1/2$)
\begin{displaymath}
\EE(Y)\le e^{\Lambda}(1-p)^{(1-\delta)N}
\le\exp\bigl\{\Lambda-(1-\delta)\,pN\bigr\}
=\exp\bigl\{\Lambda\bigl(1-64(1-\delta)\bigr)\bigr\}
\le e^{-31\Lambda}\to0,
\end{displaymath}
since $\Lambda\to\infty$. By Markov's inequality $\PP(Y>0)\to0$.

\subsection{Case 1: $1\le i\le\varepsilon_0N$}
\label{sec:case1}

The direct estimate $\PP(X_i>0)\le\EE(X_i)$ does not achieve the goal
in this range, and, as in \cite{BNR,OgRai}, we pass to \emph{maximal}
independent sets. For $i\ge1$, $j\ge i$ denote by $X_{i,j}$ the number
of maximal (with respect to inclusion) sets independent in $G_p$ of
cardinality $N-i+j$ with degree $N-i$. Let us show that
\begin{equation}
\label{eq:cover}
\Bigl\{\exists\,i\in[1,\varepsilon_0N]\colon X_i>0\Bigr\}
\subseteq\{Y>0\}\cup
\Bigl\{\exists\,i'\in[1,\varepsilon_0N],\ \exists\,j'\ge i'\colon
X_{i',j'}>0\Bigr\}.
\end{equation}
Indeed, let $D$ be independent in $G_p$, $|D|=N$,
$\deg D=N-i\le N-1$. Extend $D$ to a maximal independent in $G_p$ set
$\overline D\supseteq D$; then $|\overline D|\ge N$ and
$\deg\overline D\ge\deg D$, i.e., $\deg\overline D=N-i'$ for some
$0\le i'\le i$. If $i'=0$, then $\overline D$ contains a whole star
$S_W$; moreover $\overline D\ne S_W$ (otherwise $D\subseteq S_W$ and
$|D|=N$ imply $D=S_W$, $\deg D=N$, contrary to $i\ge1$), so that
$|\overline D|\ge N+1$ and any $(N+1)$-element subset of
$\overline D$ containing $S_W$ gives $Y>0$. If, however, $i'\ge1$,
then $\overline D$ is counted by the quantity $X_{i',j'}$ with
$j'=|\overline D|-(N-i')\ge i'$, where $i'\le i\le\varepsilon_0N$.
Inclusion \eqref{eq:cover} is proved.

The event $\{Y>0\}$ has already been treated, so it suffices to prove
that
\begin{displaymath}
\Sigma_1:=\sum_{i=1}^{\varepsilon_0N}\ \sum_{j\ge i}\EE(X_{i,j})\to0.
\end{displaymath}

Let us estimate $\EE(X_{i,j})$. To each set $D$ counted by the
quantity $X_{i,j}$ there corresponds at least one triple
$(W,A_1,A_2)$, and from the triple the set $D$ is recovered uniquely;
therefore the number of such $D$ does not exceed the number of such
triples. Here $W$ is the center of the star on which the degree is
attained ($\bC{n}{s}$ choices),
$A_1:=S_W\setminus D$, $|A_1|=i$ ($\bC{N}{i}$ choices), and
$A_2:=D\setminus S_W$, $|A_2|=j$ ($\bC{V-N}{j}$ choices). Moreover:
\begin{itemize}
\item each vertex $u\in A_2$ is joined in $G$, by
Lemma~\ref{lem:Mt}, to at least $(1-\delta)N$ vertices of the star
$S_W$, of which at most $i$ lie in $A_1$; hence from $u$ to
$D\cap S_W$ there lead at least $(1-\delta)N-i$ edges of the graph
$G$, and all of them (for distinct $u\in A_2$ these are pairwise
distinct edges) are absent in $G_p$;
\item by the maximality of $D$, each vertex $a\in A_1$ is joined in
$G_p$ to at least one vertex of the set $D$; but
$D\setminus A_2\subseteq S_W$, and a star is independent, so this
neighbor lies in $A_2$. The vertex $a$ has at most $j$ neighbors in
$A_2$, so the probability of this event does not exceed $jp$; for
distinct $a\in A_1$ these events are determined by pairwise disjoint
sets of edges and are independent of the edges of the previous item.
\end{itemize}
Consequently,
\begin{equation}
\label{eq:Xij}
\EE(X_{i,j})\le B_{i,j}:=
\bC{n}{s}\bC{N}{i}\bC{V-N}{j}\,
(1-p)^{j((1-\delta)N-i)}\,(jp)^i.
\end{equation}
For $i\le\varepsilon_0N$ and $\delta\le\varepsilon_0$ (large $n$)
\begin{equation}
\label{eq:08}
(1-\delta)N-i\ge(1-2\varepsilon_0)N=0.8\,N .
\end{equation}

Fix $i$ and let us show that $B_{i,j+1}\le\rho\,B_{i,j}$ for all
$j\ge i$, where $\rho:=\exp\{1+\ln V-0.8\,pN\}$. For $B_{i,j+1}=0$
this is trivial, and for $B_{i,j+1}>0$ (then $j<V-N$), taking
\eqref{eq:08} into account,
\begin{displaymath}
\frac{B_{i,j+1}}{B_{i,j}}
=\frac{V-N-j}{j+1}\,(1-p)^{(1-\delta)N-i}\Bigl(1+\frac1j\Bigr)^i
\le e\,V\,e^{-0.8\,pN}=\rho
\end{displaymath}
(here $(1+1/j)^i\le e^{i/j}\le e$ for $j\ge i$). By \eqref{eq:LamV} and
$pN=64\Lambda$
\begin{displaymath}
\ln\rho\le1+\ln2+\Lambda-51.2\,\Lambda\le-50\,\Lambda\to-\infty
\end{displaymath}
(the last inequality is equivalent to $1+\ln2\le0.2\,\Lambda$ and
therefore holds for $\Lambda\ge9$; recall that $\Lambda\ge r\ln\frac
nr\to\infty$). In particular, $\rho\le1/2$ for large $n$, whence by
induction
$B_{i,j}\le B_{i,i}\rho^{\,j-i}$ and
$\sum_{j\ge i}\EE(X_{i,j})\le2B_{i,i}$.

Let us estimate $B_{i,i}$ by \eqref{eq:binom}, \eqref{eq:08} and
$i\ge1$:
\begin{displaymath}
B_{i,i}\le\bC ns\Bigl[\frac{eN}{i}\cdot\frac{eV}{i}\cdot ip\cdot
e^{-0.8\,pN}\Bigr]^i
=\bC ns\Bigl[\frac{e^2\,V\,(pN)}{i}\,e^{-0.8\,pN}\Bigr]^i
\le\bC ns\,T^i,
\end{displaymath}
where $T:=e^2V\,(pN)\,e^{-0.8\,pN}$.
Since $\ln(pN)=\ln(64\Lambda)=o(\Lambda)$, for large $n$
\begin{displaymath}
\ln T\le2+\ln2+\Lambda+\ln(64\Lambda)-51.2\,\Lambda
\le-50\,\Lambda
\end{displaymath}
(the last inequality is equivalent to
$2+\ln2+\ln(64\Lambda)\le0.2\,\Lambda$ and therefore holds for
$\Lambda\ge56$); in particular $T\le1/2$. Finally, using
$\ln\bC ns\le\Lambda$ (see \eqref{eq:LamV}),
\begin{displaymath}
\Sigma_1\le2\bC ns\sum_{i\ge1}T^i\le4\bC ns\,T
\le4\exp\{\Lambda-50\Lambda\}\to0 .
\end{displaymath}

\subsection{Case 2: $\varepsilon_0N<i\le\frac{15}{16}N$}
\label{sec:case2}

Here Markov's inequality suffices. As above, to each set $D$ counted
by the quantity $X_i$ there corresponds at least one triple
$(W,A_1,A_2)$, from which it is recovered uniquely, where $W$ is the
center of the star on which the degree is attained, $A_1:=S_W\setminus D$
and $A_2:=D\setminus S_W$; since $|D|=N$ and $|D\cap S_W|=N-i$, here
$|A_1|=|A_2|=i$ (the maximality of $D$ is not assumed). Put
$d:=\deg D=N-i\ge\frac{N}{16}$. As above, from each vertex
$u\in A_2$ to $D\cap S_W$ there lead at least $(1-\delta)N-i=d-\delta N$
edges of the graph $G$, pairwise distinct for distinct $u$, and all of
them are absent in $G_p$. Since $d\ge N/16$, we have
\begin{displaymath}
d-\delta N\ge(1-\gamma)\,d,\qquad
\gamma:=16\delta=\frac{32\,r^2}{n}\to0 .
\end{displaymath}
Hence, from $p\,d\ge pN/16=4\Lambda$, the estimates \eqref{eq:binom}
and the inequality $i>\varepsilon_0N$ we obtain
\begin{displaymath}
\EE(X_i)\le\bC ns\bC Ni\bC{V-N}{i}(1-p)^{i(1-\gamma)d}
\le\bC ns\Bigl[\frac{eN}{i}\cdot\frac{eV}{i}\,
e^{-(1-\gamma)\,4\Lambda}\Bigr]^i
\le\bC ns\Bigl[\frac{e^2}{\varepsilon_0^2}\cdot\frac VN\,
e^{-3\Lambda}\Bigr]^i
\end{displaymath}
for large $n$ (when $\gamma\le1/4$). By \eqref{eq:LamV} the logarithm
of the square bracket does not exceed
\begin{displaymath}
2+2\ln\frac1{\varepsilon_0}+\bigl(\Lambda+\ln2-\ln N\bigr)-3\Lambda
\le8-2\Lambda-\ln N\le-2\Lambda
\end{displaymath}
for large $n$ (since $\ln N\ge\ln(n-r)\ge8$ by
\eqref{eq:Nnr}). Hence, summing over at most $N$ values of $i$ and
using $i>\varepsilon_0N$ and $\ln\bC ns\le\Lambda$,
\begin{displaymath}
\sum_{\varepsilon_0N<i\le\frac{15}{16}N}\PP(X_i>0)
\le\exp\Bigl\{\ln N+\Lambda-2\varepsilon_0\,\Lambda N\Bigr\}\to0,
\end{displaymath}
since $\ln N\le2r\ln n$ by \eqref{eq:binom}, while
$\Lambda N\gg r\ln n+\Lambda$.

\subsection{Case 3: $i>\frac{15}{16}N$}
\label{sec:case3}

As noted above, here $\deg D<N/16$, and by
Corollary~\ref{sled:Lstar} (for each sufficiently large $n$ either
(A) or (B) holds --- see Remark~\ref{zam:whereuslovie}) every set $D$
counted contains at least $N^2/16$ edges of the graph $G$, and all of
them are absent in $G_p$. The number of $N$-element subsets of
$V(n,r)$ does not exceed $\bC VN\le(eV/N)^N$ by \eqref{eq:binom}, so
that, using $pN^2/16=4\Lambda N$ and \eqref{eq:LamV},
\begin{align*}
\sum_{i>\frac{15}{16}N}\PP(X_i>0)
&\le\exp\Bigl\{N\Bigl(1+\ln\frac VN\Bigr)-\frac{p\,N^2}{16}\Bigr\}
\le\exp\bigl\{N\bigl(1+\ln2-\ln N+\Lambda\bigr)-4\Lambda N\bigr\}\\
&\le\exp\{-3\,\Lambda N\}\to0,
\end{align*}
since $1+\ln2-\ln N\le0$ for $N\ge6$.

\medskip
Collecting the estimates of Subsection~\ref{sec:Y} and of Cases 1--3, we
obtain $\PP\bigl(Y>0\text{ or }\exists\,i\colon X_i>0\bigr)\to0$, which
completes the proof of part (i). \hfill$\square$

\section{Proof of part (ii)}
\label{sec:lower}

By monotonicity it suffices to consider $p=p_2$; part (ii) is proved
under the assumptions $r\ge2$, $1\le s<r$, $r^2=o(n)$, which are
weaker than the hypotheses of Theorem~\ref{th:main}
(Remark~\ref{zam:iii}); in particular, $r=o(\sqrt n)$
and $\delta\to0$. Note that, according to \eqref{eq:Nnr},
\begin{displaymath}
p_2\ \le\ \frac{(r+s)\ln n}{n-r}\ \le\ \frac{4r\ln n}{n}\ \to\ 0 ;
\end{displaymath}
in particular, $p\le1/2$ for large $n$. By \eqref{eq:Lam}
\begin{equation}
\label{eq:Lambda2}
\Lambda\ \ge\ s\ln\frac ns+r\ln\frac nr-1\ \ge\
\frac{pN}{1-\varepsilon}-1,\qquad
\Lambda\ \ge\ r\ln\frac nr\to\infty .
\end{equation}

Consider the random variable $Y$ from Section~\ref{sec:upper}: the
number of sets independent in $G_p$ of the form $S_W\cup\{A\}$,
$A\notin S_W$. If $Y>0$, then $\alpha(G_p)\ge N+1$; let us prove that
$\PP(Y>0)\to1$ by the second-moment method.

Let us enumerate the admissible pairs $\pi=(W,A)$ (the center of a star
and a vertex $A\notin S_W$); their number equals
$K:=\bC{n}{s}(V-N)=e^{\Lambda}$. For a pair $\pi$ let $I_\pi$ be the
indicator of independence of the set $S_W\cup\{A\}$ in $G_p$; then
$Y=\sum_\pi I_\pi$. Indeed, every independent set of cardinality $N+1$
containing a star contains exactly one star: for $W\ne W'$
\begin{displaymath}
|S_W\cup S_{W'}|\ge2N-|S_W\cap S_{W'}|\ge2N-\bC{n-s-1}{r-s-1}>N+1,
\end{displaymath}
since $\bC{n-s-1}{r-s-1}\le Nr/n=o(N)$ by \eqref{eq:shift}, while
$N\to\infty$. Hence such a set can be written uniquely in the form
$S_W\cup\{A\}$, $A\notin S_W$, and the pairs $\pi$ with $I_\pi=1$ are
in one-to-one correspondence with the sets counted by the quantity
$Y$. A star is independent in $G$, so the independence of
$S_W\cup\{A\}$ is equivalent to the absence in $G_p$ of all edges of
the set
\begin{displaymath}
E_\pi:=\bigl\{\{A,w\}\in E(G)\colon w\in S_W\bigr\},\qquad
g_\pi:=|E_\pi|,
\end{displaymath}
and $\EE(I_\pi)=(1-p)^{g_\pi}$. By Lemma~\ref{lem:Mt} the inequalities
$(1-\delta)N\le g_\pi\le N$ hold.

\textbf{First moment.} Recall that
$L=r\ln\frac nr+s\ln\frac ns$; for $p=p_2$ we have $pN=(1-\varepsilon)L$,
and \eqref{eq:Lambda2} means $\Lambda\ge L-1$. Since $p\to0$, for large
$n$ we have $p\le\varepsilon/2$, and hence
$(1+p)(1-\varepsilon)\le1-\varepsilon+p\le1-\frac\varepsilon2$. Using
the inequality $(1-p)^N\ge e^{-pN-p^2N}$ from \eqref{eq:pp}, we obtain
\begin{equation}
\label{eq:EY}
\EE(Y)\ge K(1-p)^{N}\ge\exp\{\Lambda-(1+p)\,pN\}
\ge\exp\Bigl\{L-1-\Bigl(1-\frac\varepsilon2\Bigr)L\Bigr\}
=\exp\Bigl\{\frac{\varepsilon L}{2}-1\Bigr\}\to\infty,
\end{equation}
since $L\ge r\ln\frac nr\to\infty$.

\textbf{Second moment.} We have
\begin{displaymath}
\EE\bigl(Y(Y-1)\bigr)=\sum_{\pi\ne\pi'}
\PP\bigl(I_\pi=I_{\pi'}=1\bigr)
=\sum_{\pi\ne\pi'}(1-p)^{|E_\pi\cup E_{\pi'}|}
=\sum_{\pi\ne\pi'}(1-p)^{g_\pi+g_{\pi'}}\,
(1-p)^{-|E_\pi\cap E_{\pi'}|}.
\end{displaymath}
Let us compare this sum with
$(\EE Y)^2=\sum_{\pi,\pi'}(1-p)^{g_\pi+g_{\pi'}}$, analyzing the
intersections $E_\pi\cap E_{\pi'}$ for $\pi=(W,A)\ne\pi'=(W',A')$.

\emph{1) $A\ne A'$.} An edge $\{A,w\}\in E_\pi$ belongs to
$E_{\pi'}$ only if $\{A,w\}=\{A',w'\}$ for some $w\in S_W$,
$w'\in S_{W'}$; since $A\ne A'$, this is possible only for $w=A'$,
$w'=A$, i.e., there is at most one common edge --- the edge
$\{A,A'\}$, and only when $A'\in S_W$ and $A\in S_{W'}$. Hence,
according to \eqref{eq:pp},
\begin{displaymath}
(1-p)^{-|E_\pi\cap E_{\pi'}|}\ \le\ (1-p)^{-1}\ \le\ 1+2p ,
\end{displaymath}
and the contribution of such pairs does not exceed $(1+2p)(\EE Y)^2$.
(For $W=W'$, $A\ne A'$ we even have $E_\pi\cap E_{\pi'}=\varnothing$,
since $A,A'\notin S_W$.)

\emph{2) $A=A'$, $W\ne W'$.} The common edges have the form
$\{A,w\}$, where $w\in S_W\cap S_{W'}$, i.e., $w\supseteq W\cup W'$,
$|W\cup W'|\ge s+1$; there are at most $\bC{n-s-1}{r-s-1}\le Nr/n$ of
them by \eqref{eq:shift}. Hence, according to \eqref{eq:pp},
\begin{displaymath}
(1-p)^{-|E_\pi\cap E_{\pi'}|}\ \le\ e^{2p|E_\pi\cap E_{\pi'}|}
\ \le\ e^{2pNr/n} .
\end{displaymath}
The number of such pairs
$(\pi,\pi')$ does not exceed $\bigl(\bC ns\bigr)^2(V-N)$ (we choose the
centers $W\ne W'$ and then the vertex $A\notin S_W$), so their
contribution does not exceed
\begin{displaymath}
\bigl(\bC ns\bigr)^2(V-N)\,(1-p)^{2(1-\delta)N}\,e^{2pNr/n}
\le(\EE Y)^2\,
\frac{\exp\bigl\{4pN\delta+2pNr/n\bigr\}}{V-N},
\end{displaymath}
where we have taken into account $(\EE Y)^2\ge K^2(1-p)^{2N}$ and
$(1-p)^{-2\delta N}\le e^{4pN\delta}$. Since $pN\le2r\ln n$ and
$2\delta+\frac rn\le\frac{5r^2}{n}$, the exponent in the numerator does
not exceed
\begin{displaymath}
2pN\Bigl(2\delta+\frac{r}{n}\Bigr)
\le4r\ln n\cdot\frac{5r^2}{n}=\frac{20\,r^3\ln n}{n}
\end{displaymath}
whereas, using the inequality $V-N\ge V/2$ (Section~\ref{sec:prelim})
and \eqref{eq:binom},
\begin{displaymath}
\ln(V-N)\ge\ln\frac V2\ge r\ln\frac nr-\ln2\ge\frac r2\,\ln n-\ln2
\end{displaymath}
for large $n$ (here $r\le\sqrt n$, whence
$\ln\frac nr\ge\frac12\ln n$). Since $r^2=o(n)$, we have
\begin{displaymath}
\frac{20\,r^3\ln n}{n}=20\,\frac{r^2}{n}\cdot r\ln n
=o(r\ln n),
\end{displaymath}
and therefore
\begin{displaymath}
\frac{20\,r^3\ln n}{n}-\ln(V-N)\ \le\
-\Bigl(\frac12-o(1)\Bigr)\,r\ln n+\ln2\ \longrightarrow\
-\infty
\end{displaymath}
(since $r\ge2$). Hence the contribution of pairs of type 2) is
$o\bigl((\EE Y)^2\bigr)$.

\emph{3) $A=A'$, $W=W'$} --- excluded by the condition $\pi\ne\pi'$.

Thus,
\begin{displaymath}
\EE\bigl(Y(Y-1)\bigr)\le(1+2p)(\EE Y)^2+o\bigl((\EE Y)^2\bigr)
=(1+o(1))(\EE Y)^2,
\end{displaymath}
whence
\begin{displaymath}
\DD(Y)=\EE\bigl(Y(Y-1)\bigr)+\EE Y-(\EE Y)^2
\ \le\ o\bigl((\EE Y)^2\bigr)+\EE Y
\end{displaymath}
and by Chebyshev's inequality
\begin{displaymath}
\PP(Y=0)\le\frac{\DD(Y)}{(\EE Y)^2}\le o(1)+\frac1{\EE Y}\to0
\end{displaymath}
by \eqref{eq:EY}. Part (ii) is proved. \hfill$\square$

\section{Concluding remarks}
\label{sec:concl}

\textbf{1. The sharp constant.} Theorem~\ref{th:main} determines the
critical probability up to an absolute factor; the natural next
question is its exact value. It is clear from the proof where the
constant is lost: in Section~\ref{sec:upper} only one corollary about
families of small degree is used --- Corollary~\ref{sled:Lstar} with
constants $\beta=\eta=\frac1{16}$ (``$\deg D\le\beta N$ implies
$e(D)\ge\eta N^2$''), and the entire bound of part (i) is
$C_0\Lambda/N$ with $C_0=\max\{4/\beta,2/\eta,4\}$. Indeed, in
Case~3 the inequality $C_0\eta\ge2$ is used, in Case~2 $C_0\beta\ge4$,
and for the quantity $Y$ and Case~1 only $C_0\ge4$; the repetition of
the computations of Section~\ref{sec:upper} with arbitrary
$\beta\in(0,\frac12]$, $\eta>0$ carries over verbatim (one needs only
$C_0\Lambda=o(N)$ and $\delta/\beta\to0$). Thus, the gap between parts
(i) and (ii) equals $\asymp\max\{1/\beta,1/\eta\}$, and any improvement
of the constants in Lemma~\ref{lem:Lstar} immediately improves the
result. The constants $\beta=\eta=\frac1{16}$ in Lemma~\ref{lem:Lstar}
are certainly not optimal, but $C_0=1+o(1)$ cannot be obtained by this
method: let $\beta=1/k$ with integer $k\ge2$, let $W_1,\dots,W_k$ be
pairwise disjoint $s$-sets and $m:=\lfloor N/k\rfloor$. Take
$D=D_1\sqcup\dots\sqcup D_k$, where
$D_i\subseteq S_{W_i}\setminus\bigcup_{j\ne i}S_{W_j}$ and $|D_i|=m$;
this is possible for large $n$, since, according to \eqref{eq:shift},
\begin{displaymath}
\Bigl|S_{W_i}\setminus\bigcup_{j\ne i}S_{W_j}\Bigr|
\ \ge\ N-(k-1)\bC{n-s-1}{r-s-1}\ \ge\ N-\frac{kNr}{n}
\ =\ (1-o(1))\,N\ \ge\ m .
\end{displaymath}
The parts are pairwise disjoint, so $M:=|D|=km$, and moreover
$\deg D=m=\beta M$: exactly $m$ vertices of $D$ lie in $S_{W_i}$, while
for any center $T\notin\{W_1,\dots,W_k\}$ again by \eqref{eq:shift}
\begin{displaymath}
|D\cap S_T|\ \le\ k\,\bC{n-s-1}{r-s-1}\ \le\ \frac{kNr}{n}\ =\ o(N)
\ <\ m .
\end{displaymath}
Inside each part there are no edges, so the number of pairs of
non-neighbors is at least $k\bC m2$ and
\begin{displaymath}
e(D)\le\bC M2-k\bC m2=\frac{M^2}{2}\Bigl(1-\frac1k\Bigr)+O(M)
=(1+o(1))\,\frac{(1-\beta)M^2}{2},
\end{displaymath}
whereas $|D|=M=(1-o(1))N$ and $\deg D\le\beta M$. Hence, in
Lemma~\ref{lem:Lstar} one cannot replace $\eta$ by any constant
greater than $(1-\beta)/2$. To obtain the sharp threshold, a different
approach is apparently needed --- for $s=1$ it is known
\cite{BNR,BKL}. In this connection we also note the paper \cite{Kosh},
where the threshold for exact stability is found for the graphs
$G(n,r,s)$ (an edge is an intersection in \emph{exactly} $s$ elements)
for $r>2s+2$.

\textbf{2. Condition \eqref{eq:uslovie}.} As noted in
Remark~\ref{zam:whereuslovie}, condition \eqref{eq:uslovie} is used
only in Corollary~\ref{sled:Lstar}, and moreover only in order to
ensure the inequalities $\Phi<N/2$ and $\Psi\le N/16$; all the other
arguments of part (i) remain valid for $r^2=o(n)$, and part (ii) holds
under the same condition (Remark~\ref{zam:iii}). For $s=1$, condition
\eqref{eq:uslovie} gives $r^3=o(n)$ --- exactly the range at which the
paper \cite{BNR} stopped and which was overcome only at the cost of
invoking essentially new ideas \cite{DasTran,DevlinKahn,BKL}; for
fixed $2\le s\le r-3$ our method reaches the same barrier
(Remark~\ref{zam:uslovie}); for $s\to\infty$ the second branch of the
condition, $sr^3=o(n)$, is stricter. This suggests that further
progress requires not a refinement of the estimates, but a different
method. Note also that the bounds for $\Phi$ and $\Psi$ are
independent of each other: a refinement of either of them extends the
range of applicability of Theorem~\ref{th:main}. In particular, an
elementary (not relying on \cite{AK2}) proof of Lemma~\ref{lem:HMF} for
growing $s$ would make the whole paper self-contained;
Lemma~\ref{lem:nontriv} gives here only the bound
$\Psi\lesssim c\,r^2N/n$, i.e., it loses the factor $c$ when taking the
union over $c$ stars whose centers lie in a single vertex.

\textbf{3. Range of applicability.} In the range \eqref{eq:uslovie},
part (i) is no weaker, up to the order of magnitude, than the bounds
\cite{KOR,RaiKaras,PR}, which are polynomial in $n$, and for $r-s\to\infty$ is superpolynomially
stronger. Outside this range, for $s\ge2$ their results remain the
best known; examples in both directions are collected in
Remark~\ref{zam:compare}. For $s=1$ the complete picture is given by
the Kneser results \cite{BNR,DasTran,DevlinKahn,BKL}.

\section*{Funding}
This work was carried out without any additional funding.

\section*{Conflict of interest}
The authors declare that they have no conflicts of interest.

\end{document}